\documentclass[11pt]{amsart}

\usepackage{amsmath}               
\usepackage{amsfonts}       
\usepackage{amsthm}                
\usepackage{amssymb}
\usepackage{setspace}
\usepackage{verbatim}

\numberwithin{equation}{section}

\theoremstyle{plain}
\newtheorem{theorem}[equation]{Theorem}
\newtheorem{lemma}[equation]{Lemma}
\newtheorem{prop}[equation]{Proposition}
\newtheorem{corollary}[equation]{Corollary}

\theoremstyle{definition}
\newtheorem{definition}[equation]{Definition}
\newtheorem{example}[equation]{Example}
\newtheorem{remark}[equation]{Remark}

\newcommand{\C}{\mathbb{C}}

\newcommand{\G}{\mathbb{G}}
\newcommand{\abs}[1]{\lvert#1\rvert}
\newcommand{\Line}{\mathbb{P}^1}

\DeclareMathOperator{\Gal}{Gal}
\DeclareMathOperator{\core}{core}
\DeclareMathOperator{\Sym}{Sym}
\DeclareMathOperator{\Mon}{Mon}
\DeclareMathOperator{\charp}{char}
\DeclareMathOperator{\Aut}{Aut}

\renewcommand{\bar}[1]{#1\llap{$\overline{\phantom{\rm#1}}$}}

\usepackage[colorlinks,pagebackref,pdftex, bookmarks=false]{hyperref}

\title{On factorizations of maps between curves}

\author{Dijana Kreso}
\address{
Institut f\"ur Analysis und Computational Number Theory (Math A)\\
Technische Universit\"at Graz\\ Steyrergasse 30/II\\
8010 Graz, Austria
}
\email{kreso@math.tugraz.at}

\author{Michael E. Zieve}
\address{
  Department of Mathematics,
  University of Michigan,
  Ann Arbor, MI 48109--1043,
  USA
}
\address{Mathematical Sciences Center, Tsinghua University, Beijing 100084, China}
\email{zieve@umich.edu}
\urladdr{www.math.lsa.umich.edu/$\sim$zieve/}

\date{}

\begin{document}

\begin{abstract}
We examine the different ways of writing a cover of curves $\phi\colon C\to D$ over a field $K$ as a composition
$\phi=\phi_n\circ\phi_{n-1}\circ\dots\circ\phi_1$, where each $\phi_i$ is a cover of curves over $K$ of degree at least $2$
which cannot be written as the composition of two lower-degree covers.  We show that if the monodromy group $\Mon(\phi)$
has a transitive abelian subgroup then the sequence $(\deg\phi_i)_{1\le i\le n}$ is uniquely determined up to permutation
by $\phi$, so in particular the length $n$ is uniquely determined.  We prove analogous conclusions for the sequences
$(\Mon(\phi_i))_{1\le i\le n}$ and $(\Aut(\phi_i))_{1\le i\le n}$.
Such a transitive abelian subgroup exists in particular when $\phi$ is 
tamely and totally ramified over some point in $D(\bar{K})$, and also when $\phi$ is a morphism of one-dimensional
algebraic groups (or a coordinate projection of such a morphism).  Thus, for example,
our results apply to decompositions of
polynomials of degree not divisible by $\charp(K)$, additive polynomials, elliptic curve isogenies, and Latt\`es maps.
\end{abstract}

\thanks{The first author was supported by the Austrian Science Fund (FWF) W1230-N13 and NAWI Graz.
 The second author thanks the NSF for support under grant DMS-1162181.}

\maketitle

%################################################################################
%################################################################################
%################################################################################

\section{Introduction}

Let $\phi\colon C\to D$ be a cover of curves over a field $K$, which in this paper means a
nonconstant separable morphism between nonsingular, projective, geometrically irreducible
curves where $\phi$, $C$ and $D$ are all defined over $K$.  We will examine decompositions of $\phi$, namely expressions
$\phi=\phi_n\circ\phi_{n-1}\circ\dots\circ\phi_1$ where each $\phi_i\colon C_{i-1}\to C_i$ is a cover of curves
over $K$ with $\deg(\phi_i)\ge �2$ (so that $C_0=C$ and $C_n=D$).  Of special importance are
\emph{complete decompositions}, which are decompositions in which no $\phi_i$ can be
written as the composition of lower-degree covers.  These are the analogues of ``prime factorizations" in the context
of maps between curves.  Based on this analogy, it is natural to ask whether a given map $\phi$ has essentially just
one complete decomposition, but this turns out to be too much to hope for in general.  As a substitute, we study properties
which are shared by all complete decompositions of a given cover $\phi$.

The statements of our results involve the \emph{monodromy group} $\Mon(\phi)$ of $\phi$, which by definition is the
Galois group of (the Galois closure of) the corresponding function field extension $K(C)/K(D)$, viewed as a permutation group.  We show that
if $\Mon(\phi)$ has a transitive abelian subgroup then the sequences $(\Mon(\phi_i))_{1\le i\le n}$ and
$(\deg\phi_i)_{1\le i\le n}$ are uniquely determined (up to permutation) by $\phi$.  We prove a similar conclusion about the
sequence $(\Aut(\phi_i))_{1\le i\le n}$ of automorphism groups of the $\phi_i$'s, where by definition $\Aut(\phi_i)$ is the
group of automorphisms $\mu$ of $C_{i-1}$ defined over $K$ which satisfy $\phi_i\circ\mu=\phi_i$.  In fact we obtain these conclusions in a slightly
more general situation, as follows.

\begin{theorem} \label{mainthm}
Let $\phi\colon C\to D$ be a cover of curves over a field $K$, and let $\phi=\phi_n\circ\phi_{n-1}\circ\dots\circ\phi_1$
and $\phi=\psi_m\circ\psi_{m-1}\circ\dots\circ\psi_1$ be complete decompositions of $\phi$.  If $\Mon(\phi)$ has a transitive
quasi-Hamiltonian subgroup then $m=n$ and there is a permutation $\pi$ of $\{1,2,\dots,n\}$ such that, for each $i$ with $1\le i\le n$,
we have $\deg\phi_i=\deg\psi_{\pi(i)}$ and $\Aut(\phi_i)\cong\Aut(\psi_{\pi(i)})$.
If $\Mon(\phi)$ has a transitive Dedekind subgroup then we may choose $\pi$ so that in addition
$\Mon(\phi_i)\cong\Mon(\psi_{\pi(i)})$ for each $i$.
\end{theorem}

Here a \emph{Dedekind group} is a group $A$ such that every subgroup of $A$ is normal, and a
\emph{quasi-Hamiltonian group} is a group $A$ such that $IJ=JI$ for all subgroups $I,J$ of $A$.
Note that all abelian groups are Dedekind groups, and all Dedekind groups are quasi-Hamiltonian.
Dedekind~\cite{Dedekind} showed that the nonabelian finite Dedekind groups are precisely the direct products
of the order-$8$ quaternion group with an abelian group containing no elements of order $4$. 
Iwasawa~\cite{Iwa41} gave a similar classification of nonabelian quasi-Hamiltonian groups.

We also prove the following structural result about the automorphism group of a composition of covers.

\begin{theorem}\label{GS2intro}
Let $\theta\colon C_1\to D$ and $\rho\colon C\to C_1$ be
covers of curves over a field $K$, and assume that the monodromy group of $\psi:=\theta\circ\rho$ has a
transitive Dedekind subgroup.
Then, for each $\mu\in\Aut(\psi)$, there is a unique $\nu\in\Aut(\theta)$ for which
$\rho\circ\mu=\nu\circ\rho$.  Moreover, the map $\mu\mapsto\nu$ defines a homomorphism
$\Aut(\psi)\to\Aut(\theta)$ with kernel $\Aut(\rho)$.
\end{theorem}

As a consequence, we show in Theorem~\ref{ThmGS} that if $\phi=\phi_n\circ\phi_{n-1}\circ\dots\circ\phi_1$
then
\[
\lvert\Aut(\phi)\rvert\, \text{ divides }\,\prod_{i=1}^n\lvert\Aut(\phi_i)\rvert.
\]

We now present several classes of covers of curves which satisfy all the conclusions of the above results.
For this, it suffices to exhibit covers whose monodromy group contains a transitive abelian subgroup.
For instance, if $\phi\colon C\to D$ is totally and tamely ramified over some point 
$P\in D(\bar{K})$, then the inertia group at any point over $P$ on the Galois closure of $\phi$ 
will be a transitive cyclic subgroup of $\Mon(\phi)$.  This includes the classical case of complex polynomials, for 
which Ritt~\cite{R22} proved in 1922 that any two complete decompositions have the same length and the same collection
of degrees of the involved indecomposable polynomials.  We discuss the case of polynomials further in Section~\ref{secpoly}.
Covers of curves with a totally and tamely ramified point have arisen in other contexts as well, most recently as a distinguished
class of ``origami", meaning covers of a complex elliptic curve having a unique branch point \cite{RS}.

Our results also apply to any cover $\phi\colon C\to D$ which is the projective closure of a nonconstant separable
morphism $\phi_0\colon C_0\to D_0$
of connected one-dimensional algebraic groups.  The reason is that in this situation the transitive subgroup
$\Gal(\bar{K}(C)/\bar{K}(D))$ of $\Mon(\phi)$ is isomorphic to the kernel of $\phi_0$, and hence is abelian because
every one-dimensional algebraic group is abelian.  In case $C_0\cong D_0\cong\G_a$, the morphism $\phi$ is
an \emph{additive polynomial} $\sum_{i=0}^r a_i X^{p^i}$, where $a_i\in K$ and $p:=\charp(K)$.  Decompositions of
additive polynomials feature prominently in work on the Carlitz module and more general Drinfeld modules, see \cite{G96}.
Such decompositions were originally studied in 1933 by Ore, who proved in \cite[Thm.~4 of Chap.~2]{Ore} that any
two complete decompositions of an additive polynomial \emph{into additive polynomials}
have the same length and the same collection of degrees of the involved indecomposable polynomials.  It was shown
later that every decomposition of an additive polynomial into arbitrary polynomials is equivalent to a decomposition into
additive polynomials~\cite[Thm.~4]{DW74}, so that Ore's result strongly resembles Ritt's.  The present paper is the
first to explain this resemblance, by proving a common generalization of these two results.  Another class
of morphisms of one-dimensional algebraic groups consists of isogenies between elliptic curves.  In this case, all
portions of our results are new.

Finally, our results apply to any cover $\phi\colon C\to D$ which is a coordinate projection of a
morphism $\hat\phi\colon \hat{C}\to \hat{D}$ of one-dimensional algebraic groups.  This means that there exist nonconstant
morphisms $\pi_1\colon\hat{C}\to C$ and $\pi_2\colon\hat{D}\to D$ for which $\pi_2\circ\hat\phi=\phi\circ\pi_1$.  It was
shown in \cite{GhiocaZ} that $\Mon(\phi)$ has a transitive abelian subgroup in this situation.  This case includes the
\emph{subadditive polynomials} $S(X)\in K[X]$, which are characterized by the property that there is a positive integer $n$
for which $S(X^n)=L(X)^n$ for some additive polynomial $L(X)$.  In particular this proves the assertion from \cite[p.~325]{CHH}
that any two decompositions of a subadditive polynomial have the same length and the same degrees of the indecomposables.
Our results also apply to coordinate projections of isogenies of elliptic curves, which play a prominent role in the finite fields
literature \cite{GMS,Mullerec}; in case the isogeny is an endomorphism, such coordinate projections are called \emph{Latt\`es maps}
and play a crucial role in complex dynamics \cite{Milnor}.

Since the transitive abelian subgroup in each of the above cases is a subgroup of the geometric monodromy group of $\phi$ -- that is,
the monodromy group of the base extension of $\phi$ to a morphism of curves over $\bar{K}$ -- it follows that our results also
apply to any cover which becomes isomorphic to one of the above covers after base extension to $\bar{K}$.  For instance, this
includes decompositions of Dickson polynomials \cite{LMT}, which are quadratic twists of Chebyshev polynomials (which in turn are coordinate
projections of the multiplication-by-$d$ endomorphism of $\G_m$). It also includes R\'edei functions \cite{Redei}, which are rational functions
inducing covers $\Line\to\Line$ that become isomorphic to $X^d$ over a quadratic extension of $K$.

In the development that follows, we also prove several other results about decompositions.  In some cases we give simpler proofs
(in greater generality) of results from previous papers: for instance one can compare Corollary~\ref{corfin} and Remark~\ref{sillyrem}, or
Lemma~\ref{symlem} and the last paragraph of Section~\ref{Sec:BNg}.  Also in Remark~\ref{finalsay} we disprove a conjecture from
\cite{GS06}.  These improvements on previous work are made possible in part by our generalization to the framework of covers of curves, which
provides a valuable perspective even when one is only interested in
questions about polynomials.

This paper is organized as follows.  In the next section we explain the connection between monodromy groups and
decompositions of a cover.  In Section~\ref{secpoly} we expand on this
connection in the much-studied case of decompositions of polynomials.  In Sections~\ref{Sec:Quasi}, \ref{Sec:Mon} and \ref{Sec:BNg}
 we prove the portions of
Theorem~\ref{mainthm} pertaining to degrees, monodromy groups, and automorphism groups, respectively.  We conclude in
Section~\ref{secdiv} by proving Theorem~\ref{GS2intro}.

%################################################################################
%################################################################################

\section{Decomposition of covers via monodromy groups} \label{Sec:Prelim}

In this section we translate the problem of analyzing decompositions of a cover of curves
into the problem of analyzing chains of subgroups of its monodromy group, which we then
reformulate as analyzing chains of certain types of subgroups of a
transitive subgroup of this monodromy group.

We first introduce the terminology we will use in the paper.

\begin{definition}A \emph{curve} over a field $K$ is a nonsingular, projective, geometrically
irreducible one-dimensional variety defined over $K$.
\end{definition}

\begin{definition}A \emph{cover} of curves over a field $K$ is a nonconstant separable morphism
between curves over $K$.
\end{definition}

\begin{definition}A cover of curves over $K$ is \emph{decomposable} if it can be written
as the composition of two covers (of curves over $K$) which both have degree at least $2$.  A cover is
\emph{indecomposable} if its degree is at least $2$ and it is not decomposable.
\end{definition}

\begin{definition}A \emph{decomposition} of a cover $\phi\colon C\to D$ over $K$ is an expression
$\phi=\phi_n\circ\phi_{n-1}\circ\dots\circ\phi_1$ where each $\phi_i\colon C_{i-1}\to C_i$ is a cover
(of curves over $K$) of degree at least $2$.  Such a decomposition is a
 \emph{complete decomposition} if every $\phi_i$ is indecomposable.
\end{definition}

\begin{definition} Let 
$\phi=\phi_n\circ\phi_{n-1}\circ\dots\circ\phi_1=\psi_m\circ
\psi_{m-1}\circ\dots\circ\psi_1$ be two decompositions of a cover $\phi\colon C\to D$, where
$\phi_i\colon C_{i-1}\to C_i$ and $\psi_i\colon B_{i-1}\to B_i$ (and $C_0=B_0=C$
and $C_n=B_n=D$).  We call these decompositions \emph{equivalent}
if $m=n$ and there are isomorphisms $\rho_i\colon C_i\to B_i$ such that 
$\rho_0\colon C\to C$ and $\rho_n\colon D\to D$ are the identity maps and
$\psi_i\circ \rho_{i-1}=\rho_i\circ\phi_i$ for $1\le i\le n$.
\end{definition}

\begin{definition}The \emph{monodromy group} $\Mon(\phi)$ of a cover $\phi\colon C\to D$ of
curves over $K$
is the Galois group of the Galois closure of the extension of function fields $K(C)/K(D)$.
\end{definition}

We view $\Mon(\phi)$ as a group of permutations of the set of embeddings of
$K(C)$ into a fixed algebraic closure of $K(D)$ which restrict to the identity map on $K(D)$.
The number $d$ of such embeddings is $[K(C):K(D)]=\deg\phi$, so that $\Mon(\phi)$ is a
subgroup of $S_d$, and further $\Mon(\phi)$ is transitive.

\begin{example}
We illustrate the above notions in the special case of covers $\phi\colon\Line\to\Line$:
upon choosing coordinates on both copies of $\Line$, we see that such a cover is the same thing as
a rational function $f(X)\in K(X)$ with nonzero derivative, or equivalently an element of
$K(X)\setminus K(X^p)$ where $p:=\charp(K)$.  Then two decompositions
$f=f_n\circ f_{n-1}\circ\dots\circ f_1$ and $f=g_m\circ g_{m-1}\circ\dots\circ g_1$ are equivalent if $m=n$
and there are degree-one $\mu_i\in K(X)$ (for $0\le i\le n$) such that $\mu_0=\mu_n=X$ and
$g_i\circ\mu_{i-1}=\mu_i\circ f_i$ for $1\le i\le n$.  In this case $\Mon(\phi)$ is the Galois group of
the numerator of $f(X)-t$ over $K(t)$, where $t$ is transcendental over $K$.
\end{example}

Having defined our terminology, we now state our first result.

\begin{lemma}\label{Decr.Incr}
Let $K$ be a field and let $\phi\colon C\to D$ be a cover of curves over $K$.
Let $G$ be the monodromy group of $\phi$, let $H$ be a one-point stabilizer in $G$, and let $A$
be a transitive subgroup of $G$.  There are bijections between the following sets:
\begin{enumerate}
\item[(1)] the set of equivalence classes of decompositions of $\phi$,
\item[(2)] the set of increasing chains of fields between $K(D)$ and $K(C)$,
\item[(3)] the set of decreasing chains of groups between $G$ and $H$,
\item[(4)] the set of decreasing chains of groups between $A$ and $H\cap A$ consisting of groups $J$
for which $JH=HJ$.
\end{enumerate}
Moreover, these bijections can be chosen so that the degrees of the indecomposable covers in a
decomposition in \emph{(1)}
equal the indices between successive groups in the corresponding chain in each of \emph{(3)} and \emph{(4)}.
\end{lemma}

\begin{proof}
Let $\phi=\phi_n\circ\phi_{n-1}\circ\dots\circ\phi_1$ be a decomposition of $\phi$, where $\phi_i\colon C_{i-1}\to C_i$
with $C_0=C$ and $C_n=D$.
Associate to this decomposition the chain of fields $K(C_n)\subset K(C_{n-1})\subset\dots\subset K(C_0)$,
where the inclusion $K(C_i)\hookrightarrow K(C_{i-1})$ is defined by $\psi\mapsto \psi\circ\phi_i$.
Then the usual equivalence of categories between curves over $K$ (and their covers) and
finitely generated field extensions of $K$ having transcendence degree $1$ (and their separable
extensions) \cite[Cor.~6.12]{Hartshorne} shows that this association yields a bijection between (1) and (2),
and also that $\deg\phi_i=[K(C_{i-1}):K(C_i)]$.  Next let $\Omega$ be the Galois closure of
$K(C)/K(D)$, so that $G=\Gal(\Omega/K(D))$.  Then the Galois correspondence \cite[Thm.~VI.1.1]{Lang}
 yields a bijection
between (2) and the set of decreasing chains of groups between $G$ and $\tilde H:=\Gal(\Omega/K(C))$,
where the degree of each successive extension in the chain of fields equals the index between the 
corresponding groups in the chain of groups.
Since $\tilde H$ and $H$ are conjugate subgroups of $G$, this yields a bijection between (2) and (3).
The following lemma yields a bijection between (3) and (4), and shows that the indices between
successive groups in a chain in (4) equal the analogous indices in the corresponding chain in (3).
\end{proof}

\begin{lemma} \label{G1}
Let $G$ be a permutation group, let $H$ be a one-point stabilizer, and let $A$ be a transitive subgroup
of $G$.  Then the map $\rho\colon U\mapsto U\cap A$ is a bijection from the set of groups between
$G$ and $H$  to the set of groups $J$ between $A$ and $H\cap A$ for which $JH=HJ$.
Moreover, $[G:U]=[A:U\cap A]$ and  $\rho(\langle U,V\rangle)=\langle\rho(U),\rho(V)\rangle$
and $\rho(U\cap V)=\rho(U)\cap\rho(V)$ for any groups $U,V$ between $H$ and $G$.
\end{lemma}

\begin{proof}
Transitivity of $A$ means that $G=AH$.  Every group $U$ between $H$ and $G$ is a union of cosets $gH$
with $g\in G$, and since $G=AH$ we know that every such coset equals $aH$ for some $a\in A$, whence
$U=(U\cap A)H$.
  Conversely, if $J$ is a group between
$A$ and $H\cap A$ then $JH$ is a group if and only if $JH=HJ$ (in which case $[G:JH]=[A:J]$).
This proves that $\rho$ is a bijection.  The final assertion follows from bijectivity of $\rho$
and the fact that if $I,J$ are groups between $A$ and $H\cap A$  which satisfy $IH=HI$ and $JH=HJ$ 
then also $\langle I,J\rangle H=H \langle I,J\rangle$ and 
$IH\cap JH=(I\cap J)H$, whence $(I\cap J)H=H(I\cap J)$.
\end{proof}

The utility of Lemma~\ref{Decr.Incr} stems from the fact that, for any fixed positive integer $n$,
questions about an infinite collection of objects (namely, all degree-$n$ covers of curves over
an arbitrary field $K$) have been translated into questions about a finite collection of objects
(namely, certain types of subgroups of $S_n$).  One immediate consequence of this translation
is as follows.

\begin{corollary} \label{corfin}
Any cover of curves over any field $K$ has only finitely many equivalence classes of 
decompositions.  Moreover, the number of such equivalence classes of decompositions is bounded
above by a constant which depends only on the degree of the cover.
\end{corollary}

\begin{proof}
By Lemma~\ref{Decr.Incr}, the number of equivalence classes of decompositions of a degree-$n$
cover is at most the number of decreasing chains of subgroups of $S_n$.
\end{proof}

\begin{remark} \label{remprim}
In the case of covers $\Line\to\Line$, Corollary~\ref{corfin} asserts that a rational function
in $K(X)$ with nonzero derivative has only finitely many equivalence classes of decompositions.
In fact the proof of Lemma~\ref{Decr.Incr} implies the same conclusion for rational functions with
zero derivative, since the number of equivalence classes of decompositions equals the number of
decreasing chains of fields between $K(x)$ and $K(f(x))$ (where $x$ is transcendental over $K$),
and there are only finitely many fields between $K(x)$ and $K(f(x))$ since $K(x)/K(f(x))$ is a simple
extension \cite[Thm.~V.4.6]{Lang}.  However, there exist inseparable finite morphisms between curves
which admit infinitely many equivalence classes of decompositions \cite[Exerc.~V.24]{Lang}.
This behavior is typical for questions about decompositions of inseparable morphisms:
inseparable morphisms between curves can have completely different properties than do separable
morphisms, but inseparable rational functions behave in exactly the same way
as do separable rational functions (and likewise for polynomials).
\end{remark}

\begin{remark}
In the case of covers $\Line\to\Line$, Lemma~\ref{Decr.Incr} may be compared with the assertion
in \cite[p.~666]{GS2} that in this setting ``there is no bijection between groups and intermediate fields".
\end{remark}

%################################################################################
%################################################################################

\section{Functional decomposition of polynomials} \label{secpoly}

Functional decomposition is often studied for polynomials $f(X)\in K[X]$, where one is interested in
the expressions of $f(X)$ as the composition of polynomials in $K[X]$ of degree at least $2$.  Here we
say that two decompositions 
$f=f_n\circ f_{n-1}\circ\dots\circ f_1$ and $f=g_m\circ g_{m-1}\circ\dots\circ g_1$ are equivalent if $m=n$
and there are degree-one $\mu_i\in K[X]$ (for $0\le i\le n$) such that $\mu_0=\mu_n=X$ and
$g_i\circ\mu_{i-1}=\mu_i\circ f_i$ for $1\le i\le n$.  Note that we have already defined a different
notion of decompositions of a polynomial, since a polynomial may be viewed as a rational function.
In this section we show that these two notions are compatible, and we also show that if
$\charp(K)\nmid\deg(f)$
then the monodromy group of $f(X)$ has a transitive cyclic subgroup.

\begin{lemma} \label{ratpol}
Let $K$ be a field and pick any $f(X)\in K[X]$.  Then 
every equivalence class of decompositions of $f(X)$ in the sense of rational functions contains 
exactly one equivalence class of decompositions of $f(X)$ in the sense of polynomials.
\end{lemma}

\begin{proof}
Rational functions (or polynomials) of degree less than $2$ have no decompositions according to our
definitions, so we may assume that $\deg(f)\ge 2$.  Write $f=f_n\circ f_{n-1}\circ\dots\circ f_1$
where $f_i\in K(X)$.
Since $f$ is a polynomial, we have $f^{-1}(\infty)=\{\infty\}$, so that $\infty$ has a unique preimage
under $f_n\circ\dots\circ f_i$ whenever $1\le i\le n$.  Define $\mu_n=\mu_0=X$ and,
for each $i=n-1,n-2,\dots,1$ in succession, let $\mu_i\in K(X)$ be a degree-one rational function
for which $\hat f_{i+1}:=\mu_{i+1}^{-1}\circ f_{i+1}\circ\mu_i$ fixes $\infty$, and hence is a polynomial.
  Then also $\hat f_1:=\mu_1^{-1}\circ f_1\circ\mu_0$ fixes $\infty$, and
$f=\hat f_n\circ\hat f_{n-1}\circ\dots\circ\hat f_1$ is a decomposition of $f$ which is equivalent
to our original decomposition.  Finally, our procedure for defining the $\mu_i$'s shows that any
two choices yield decompositions which are equivalent in the sense of polynomials.
\end{proof}

In light of this result, Remark~\ref{remprim} implies the following:

\begin{corollary} \label{silly}
Any $f(X)\in K[X]$ has only finitely many equivalence classes of decompositions
in the sense of polynomials.
\end{corollary}

\begin{remark} \label{sillyrem}
 The special case of this result for $K=\C$ was first proved in 1922
\cite[\S 2]{R22} via essentially the same method as above.
Corollary~\ref{silly} also follows at once from \cite[\S 2]{Levi} and \cite[Cor.~2.3]{vzGtame},
where a method is used which
only applies to polynomials.  After appearing in dozens of papers and books  over
the next several decades, the case $K=\C$ of Corollary~\ref{silly} arose again in 2000 as
one of the main ``new" results of \cite{BNg}, where it was proved by a more complicated version
of the proof in \cite[\S 2]{Levi}.  The authors of \cite{BNg} motivated the $K=\C$ case of
Corollary~\ref{silly} by making the curious assertion that no previous authors had
noticed the special role of degree-one polynomials in the theory of functional decomposition;
however, this special role is addressed in nearly every treatment of this topic,
for instance \cite{Binder,DW74,FM69,vzGtame,LauschN,Levi,R22,Schinzelold}.  Indeed, this is an
instance of the special role units play in the theory of factorization in any monoid.
\end{remark}

Next we show that, if $f(X)\in K[X]$ has degree not divisible by $\charp(K)$,
then the monodromy group of $f(X)$ contains
a transitive cyclic subgroup.  Here the monodromy group is just the Galois group of $f(X)-t$
over $K(t)$, where $t$ is transcendental over $K$.  One such transitive cyclic subgroup is the
inertia group at any place of the splitting field of $f(X)-t$ which lies over the infinite place of $K(t)$,
as has been well-known for over a hundred years.  For the benefit of authors unfamiliar
with inertia groups, we include here a self-contained proof of the existence of a transitive
cyclic subgroup (based on Newton's ideas as arranged in \cite[Lemma~3.3]{Turnwald}).

\begin{lemma} \label{noinertia}
If $f(X)\in K[X]$ is a degree-$n$ polynomial over a field $K$ for which $\charp(K)\nmid n$,
then the monodromy group of $f(X)$ contains a transitive cyclic subgroup.
\end{lemma}

\begin{proof}
Let $t$ be transcendental over $K$, and let $\bar{K}$ be an algebraic closure of $K$.
Let $L:=\bar{K}((1/t))$ be the field of formal Laurent series over $\bar{K}$,
namely the set of expressions $\sum_{i=-\infty}^{\infty} a_i t^i$ where $a_i\in \bar{K}$ and where 
in addition there exists an integer $N$ for which $a_i=0$ whenever $i>N$.
We first show that the Galois group of $f(X)-t$ over $L$ has a transitive cyclic subgroup.
Let $s$ be any root of $X^n-t$ in an extension of $L$, and note that $L(s)=\bar{K}((1/s))$.
Let $b$ be the leading coefficient of $f(X)$.
For any $c\in\bar{K}$ such that $c^n=b$, if we write
$x_c:=c s+\sum_{i=-\infty}^{0} a_i s^i$
then there is a unique choice of elements $a_i\in\bar{K}$ for which $f(x_c)=s^n$,
since for each $i=0,1,2,3,\dots$ in succession we can uniquely determine $a_i$ from
the condition that the coefficient of $s^{n-1-i}$ in $f(x_c)-s^n$ equals zero.
Now let $\theta\in\bar{K}$ be a primitive $n$-th root of unity, and let
$\sigma$ be the automorphism of $L(s)$ which maps $\sum_i a_i s^i$ to $\sum_i \sigma(a_i) s^i$.
Since $\sigma$ fixes every element of $L$, it must permute the roots of $f(X)-t$, namely the $n$
elements $x_c$ with $c^n=b$.  By considering the action of $\sigma$
on the coefficient of $s$ in the various elements $x_c$, we see that $\sigma$ induces a transitive 
permutation on the $x_c$'s.  Therefore $\langle\sigma\rangle$ is a transitive cyclic subgroup
of the Galois group of $f(X)-t$ over $L$, so the restriction of $\langle\sigma\rangle$ to the splitting
field of $f(X)-t$ over $K(t)$ is a transitive cyclic subgroup of the monodromy group of $f(X)$.
\end{proof}

%################################################################################
%################################################################################

\section{Ritt's first theorem}\label{Sec:Quasi}

In this section we show that if $\phi\colon C\to D$ is a cover of curves whose monodromy group
has a transitive quasi-Hamiltonian subgroup, then any two complete decompositions of $\phi$ have
the same length and the same multiset of degrees of the involved indecomposable subcovers.
Moreover, we prove that we can pass from any complete decomposition of $\phi$ to any other via
finitely many steps of a special form; this will play a crucial role in subsequent sections.

\begin{theorem}\label{RittFirst}
Let $\phi\colon C\to D$ be a cover of curves over a field $K$, and suppose that the
monodromy group of $\phi$ has a transitive quasi-Hamiltonian subgroup.
Then any complete decomposition $\phi=\phi_n\circ\phi_{n-1}\circ\dots\circ\phi_1$ can be
obtained from any other complete decomposition
$\phi=\psi_m\circ\psi_{m-1}\circ\dots\circ\psi_1$ through finitely many steps, where in each
step we replace two adjacent indecomposable covers $\theta_2,\theta_1$ in a complete
decomposition by two
other indecomposable covers $\hat \theta_2,\hat\theta_1$ 
such that $\theta_2\circ\theta_1=\hat\theta_2\circ\hat\theta_1$ and
$\{\deg\theta_1,\deg\theta_2\}=\{\deg\hat\theta_1,\deg\hat\theta_2\}$.
In particular, $m=n$ and the sequence $(\deg\phi_i)_{1\le i\le n}$ is a permutation of the
sequence $(\deg\psi_i)_{1\le i\le m}$.
\end{theorem}

By Lemma~\ref{Decr.Incr}, Theorem~\ref{RittFirst} is a consequence of the following
group-theoretic assertion.

\begin{lemma}\label{chains}
Let $A$ and $B$ be finite groups with $B\le A$.  Let $S$ be a
set of groups between $A$ and $B$ such that $A,B\in S$ and $S$ contains both
 $IJ$ and $I\cap J$ whenever it contains groups $I$ and $J$.
Let $A=V_n>V_{n-1}> \dots > V_{0}=B$ and
 $A=W_m>W_{m-1}> \dots > W_{0}=B$ be two maximal decreasing chains of groups in $S$
 which lie between $A$ and $B$.
 Then one can pass from the first chain to the second chain through finitely many steps, where
  in each step a chain $A=C_k>C_{k-1}>\cdots >C_{0}=B$ is replaced by a chain
   $A=D_k>D_{k-1}>\cdots >D_0=B$ of groups
 in $S$ such that $D_i=C_i$ for all but one $i$ with $0<i<k$.  Moreover, if $C_i\ne D_i$ then
  $[C_i:C_{i-1}]=[C_{i+1}:D_i]$.
\end{lemma}

\begin{proof}
We prove the result by induction on $\abs{A}$, noting that the result is vacuously true when $\abs{A}=0$.
Let $A$ be a finite group, and assume that the assertion holds for all
groups of order less than $\abs{A}$.  Pick any $B$ and $S$ as in the lemma, and
let $A=V_n>V_{n-1}> \dots > V_{0}=B$ and
 $A=W_m>W_{m-1}> \dots > W_{0}=B$ be two maximal increasing chains of groups in $S$.
If $V_{n-1}=W_{m-1}$ then the inductive hypothesis implies that the
chains $V_{n-1}>V_{n-2}>\dots>V_0$ and
$W_{m-1}>W_{m-2}>\dots>W_{0}$ satisfy the desired conclusion, so 
the desired conclusion also holds for
the two chains obtained by appending $A$ to both of these chains.
Henceforth we assume that $V_{n-1}\ne W_{m-1}$.  Then $V_{n-1}W_{m-1}$ is a group in $S$
which is strictly larger than at least one of $V_{n-1}$ or $W_{m-1}$, so the maximality of the chains
implies that $V_{n-1}W_{m-1}=A$, whence $[A:W_{m-1}]=[V_{n-1}:V_{n-1}\cap W_{m-1}]$.
 Let $U$ be a group in $S$ such that
$V_{n-1}\geq U>V_{n-1}\cap W_{m-1}$.  Then $S$ contains $UW_{m-1}$, and
$[UW_{m-1}:W_{m-1}]=[U:U\cap W_{m-1}]=[U:V_{n-1}\cap W_{m-1}]>1$.  Maximality of the chains
implies that $UW_{m-1}=A$, so that $[A:W_{m-1}]=[U:V_{n-1}\cap W_{m-1}]$, whence
$U=V_{n-1}$.  Now let $V_{n-1}\cap W_{m-1}=Y_r>\dots>Y_0=B$
be any maximal chain of groups in $S$ which lie between  $V_{n-1}\cap W_{m-1}$ and $B$.
Appending $V_{n-1}$ yields a maximal chain of groups in $S$ which lie between $V_{n-1}$ and
$B$, so by inductive hypothesis we can pass from this chain to the chain
$V_{n-1}>\dots>V_0=B$ by steps of the required type.  Therefore if we augment both chains
by appending $A$, we can still pass between these augmented chains via steps of the required
type.  The same argument shows that steps of the required type enable us to pass from
$A>W_{m-1}>\dots>W_0$ to $A>W_{m-1}>Y_r>\dots>Y_0$, which implies the
desired conclusion since the replacement of $A>W_{m-1}>Y_r>\dots>Y_0$ by
$A>V_{n-1}>Y_r>\dots>Y_0$  is a step of the required type.
\end{proof}

\begin{remark}
Our proof of Theorem~\ref{RittFirst} actually shows something slightly stronger, since we do not
need the monodromy group $G$ of $\phi$ to contain a transitive quasi-Hamiltonian subgroup.
What we actually need is that, if $H$ is a one-point stabilizer of $G$, then $G$ contains a
transitive subgroup $A$ with the property that $IJ=JI$ for all groups $I,J$ such that
$H\cap A\le I,J\le A$ and $HI=IH$ and $HJ=JH$. We do not know whether there are any
 natural situations in which Theorem~\ref{RittFirst} does not apply but this stronger version does.
\end{remark}

\begin{remark}\label{Rittrem}
In case $\phi$ is given by a polynomial $f(X)\in K[X]$ of degree not divisible by $\charp(K)$,
the monodromy group of $\phi$ contains a transitive cyclic subgroup by Lemma~\ref{noinertia},
so the conclusion of Theorem~\ref{RittFirst} holds.
In this case Theorem~\ref{RittFirst} is known as Ritt's First Theorem, and it was first proved by Ritt
for $K=\C$ \cite{R22}.  A different proof was given by
Engstrom \cite[Thm.~4.1]{Engstrom} in case $K$ is an arbitrary field of characteristic
zero, and Engstrom's proof extends at once to polynomials over any field with
$\charp(K)\nmid\deg(f)$ (cf.\
\cite[Thm.~4.1.34]{LauschN},
 \cite[Thm.~7]{Schinzel}, \cite[Thm.~5.11]{Binder}, \cite[Thm.~VII.5]{Wyman}).
Ritt's proof may be viewed as a special
case of the proof given above, although it is presented in a different language.
Alternate versions of Ritt's proof are given in \cite[Thm.~3.1]{FM69}
and \cite[Thm.~2]{DW74} for polynomials of degree less than $\charp(K)$
(see also \cite[Thm.~2.1 and Cor.~2.12]{ZM}).
Yet another version of Ritt's proof for polynomials in characteristic zero is given in \cite[Thm.~R.1]{Muller},
where it is noted that the transitive cyclic subgroup used in Ritt's proof can be replaced by a
transitive abelian subgroup.  A slightly weaker version of Theorem~\ref{RittFirst} is stated in
\cite[Cor.~1.5]{KLZ}.
\end{remark}

%################################################################################
%################################################################################

\section{The monodromy invariant}\label{Sec:Mon}

In the previous section we showed that, if $\phi$ is a cover of curves whose monodromy group
$\Mon(\phi)$
has a transitive quasi-Hamiltonian subgroup, then any two complete decompositions of $\phi$ have
the same length and the same collection of degrees of the involved indecomposable subcovers.
In this section we show that a stronger conclusion holds under a slightly more restrictive hypothesis:
specifically, if $\Mon(\phi)$ has a transitive {Dedekind} subgroup then any two complete 
decompositions
of $\phi$ have the same collection of monodromy groups of the involved indecomposable subcovers.
Here, as usual, the monodromy groups are viewed as permutation groups,  so that the degree of the
monodromy group equals the degree of the corresponding cover, and hence covers with isomorphic
monodromy groups have the same degrees as one another.

The main result of this section is as follows.

\begin{theorem}\label{MonThmfull}
Let $\phi\colon C\to D$ be a cover of curves over a field $K$, and suppose that the monodromy
group $\Mon(\phi)$  has a transitive Dedekind subgroup.
If $\phi=\phi_n\circ\phi_{n-1}\circ\dots\circ\phi_1$ and
$\phi=\psi_n\circ\psi_{n-1}\circ\dots\circ\psi_1$ are complete decompositions of $\phi$,
then there is a permutation $\pi$ of $\{1,2,\dots,n\}$ such that, for each $i=1,2,\dots,n$,
the groups $\Mon(\phi_i)$ and $\Mon(\psi_{\pi(i)})$ are isomorphic as permutation groups.
\end{theorem}

\begin{remark} As noted above, covers with isomorphic monodromy groups must have the same degree.
We will show in Lemma~\ref{symlem} that two such covers must also have isomorphic automorphism groups,
so that Theorem~\ref{mainthm} follows from Theorem~\ref{MonThmfull} if $\Mon(\phi)$ has a transitive
Dedekind subgroup.
\end{remark}

In light of Theorem~\ref{RittFirst}, it suffices to prove Theorem~\ref{MonThmfull} when $n=2$.
In fact we will prove the following refinement of the case $n=2$ of Theorem~\ref{MonThmfull}:

\begin{prop}\label{MonThm}
Let $\phi\colon C\to D$ be a cover of curves over a field $K$, and suppose that the monodromy
group $\Mon(\phi)$ has a transitive Dedekind subgroup.  If $\phi=\phi_2\circ\phi_1$ and
$\phi=\psi_2\circ\psi_1$ are inequivalent complete decompositions of $\phi$, then
$\Mon(\phi_2)\cong\Mon(\psi_1)$ (as permutation groups) and likewise
 $\Mon(\phi_1)\cong\Mon(\psi_2)$.\end{prop}

In order to prove Proposition~\ref{MonThm}, we first translate it into a group-theoretic statement.
This requires the following terminology.

\begin{definition}
If $W$ is a subgroup of a group $G$, then the \emph{core} of $W$ in $G$ is the largest normal
subgroup of $G$ which is contained in $W$, and is denoted $\core_G(W)$.
\end{definition}

\begin{remark}
Two basic properties of cores are as follows: first, 
$\core_G(W)=\cap_{g\in G}W^g$, where $W^g:=g^{-1}Wg$.  Second, $\core_G(W)$ is the kernel
of the homomorphism $G\to\Sym(G/W)$ induced by the action of $G$ by left multiplication on the set
$G/W$ of left cosets of $W$ in $G$.
\end{remark}

Next we use cores to describe the monodromy groups of the subcovers occurring in a decomposition
of a cover.

\begin{lemma}\label{coremon}
Let $\phi\colon C\to D$ be a cover of curves over a field $K$, and let $\phi=\phi_n\circ\phi_{n-1}\circ
\dots\circ\phi_1$ be a decomposition of $\phi$, where $\phi_i\colon C_{i-1}\to C_i$.
Write $G_i:=\Gal(\Omega/K(C_i))$, where $\Omega$ is the Galois closure of $K(C)/K(D)$.
Then $\Mon(\phi_i)$ is isomorphic as a permutation group to the group $G_i/\core_{G_i}(G_{i-1})$
in its action on the set of left cosets of $G_{i-1}$ in $G_i$.
\end{lemma}

\begin{proof}
By definition, $\Mon(\phi_i)$ is the Galois group of the Galois closure of $K(C_{i-1})/K(C_i)$, and hence
as an abstract group $\Mon(\phi_i)\cong G_i/\core_{G_i}(G_{i-1})$.  We view $\Mon(\phi_i)$ as a group of
permutations of the set $\Lambda$ of homomorphisms $K(C_{i-1})\to\bar{K(C_i)}$ which 
restrict to the identity on $K(C_i)$.  The image of any such homomorphism is contained in $\Omega$ (since
$\Omega/K(C_i)$ is normal), so we can identify $\Lambda$ with $G_i/G_{i-1}$ without changing the action
of $\Mon(\phi_i)$.
\end{proof}

In combination with Lemma~\ref{Decr.Incr}, this lemma shows that Proposition~\ref{MonThm}
is a consequence of the following result.

\begin{prop}\label{MonProp}
Let $G$ be a permutation group with a transitive Dedekind subgroup $A$,
and let $H$ be a one-point stabilizer of $G$.  If $G>U>H$ and $G>V>H$ are distinct maximal decreasing
chains of groups between $G$ and $H$, then
$G/\core_G(U)$ (in its action as a permutation group on $G/U$) is isomorphic to $V/\core_V(H)$ 
(in its action as a permutation group on $V/H$).
\end{prop}

The following lemma exhibits the portion of Proposition~\ref{MonProp} which we can prove under
the weaker hypothesis that $G$ has a transitive quasi-Hamiltonian subgroup.

\begin{lemma}\label{MonLemma}
Let $G$ be a permutation group which has a transitive quasi-Hamiltonian subgroup $A$,
and let $H$ be a one-point stabilizer of $G$.  If $G>U>H$ and $G>V>H$ are distinct maximal decreasing
chains of groups between $G$ and $H$, then $N:=\core_G(U)$ and
$C:=\core_V(H)$ satisfy either
$G/N\cong V/C$ or $N=C=\core_G(H)$.
\end{lemma}

\begin{proof}
By Lemma~\ref{G1} we have $U=HI=IH$ and $V=HJ=JH$ where $I:=U\cap A$ and $J:=V\cap A$,
and also $U\cap V=H(I\cap J)$ and $\langle U,V\rangle=H\langle I,J\rangle$.
Maximality (and distinctness) of the chains implies that $U\cap V=H$ and $\langle U,V\rangle=G$,
so that $I\cap J=H\cap A$ and $\langle I,J\rangle=A$.  Since $A$ is quasi-Hamiltonian we have
$IJ=\langle I,J\rangle=A$, and since $HI=IH$ and $HJ=JH$ we find that $UV=HIHJ=HIJ=HA=G$.
The facts that $U\cap V=H$ and $UV=G$ imply that
\begin{align*}
C=\bigcap_{g\in V} H^{g}&=\bigcap_{g\in V} (U\cap V)^{g}=\bigl(\bigcap_{g\in V} U^{g}\bigr) \cap V\\
&=\bigl(\bigcap_{g\in G} U^{g}\bigr) \cap V=N\cap V.
\end{align*}
Since $N$ is normal in $G$, we know that $NV$ is a subgroup of $G$,
so maximality of the chain $G>V>H$ implies that either $NV=V$ or $NV=G$.
If $NV=V$ then $V\ge N$ so $C=N\cap V=N$, whence $C=N=\core_G(H)$.
  Finally, suppose that $NV=G$.  
    Since $N\cap V=C$ is a normal subgroup of $V$ and $NH=U$ (by maximality),
the natural map $V/(N\cap V)\to NV/N$ is an isomorphism of permutation groups $V/C\cong G/N$.
 \end{proof}

We now prove Proposition~\ref{MonProp}, which as we have seen implies Proposition~\ref{MonThm}
and Theorem~\ref{MonThmfull}.  In the notation of Lemma~\ref{MonLemma}, all that must be shown is
that if $A$ is Dedekind then $N\ne C$.

\begin{proof}[Proof of Proposition~\ref{MonProp}]
We may assume that $\core_G(U)=\core_V(H)=\core_G(H)$, since otherwise
Lemma~\ref{MonLemma} implies the desired conclusion.
Since $G=HA$ and $U>H$ we have $G=UA$, so that
\[
\core_G(U)=\bigcap_{g \in G} U^g= \bigcap_{g \in A} U^g \geq \bigcap_{g\in A} (U \cap A)^g = 
\core_A(U\cap A).
\]
But $\core_A(U\cap A)=U\cap A$ (since $A$ is Dedekind), 
so
$U\cap A\le\core_G(U)=\core_G(H)$, which yields the contradiction
$U=H(U\cap A)\le H$.
\end{proof}

\begin{remark}
We do not know whether Theorem~\ref{MonThmfull} and Proposition~\ref{MonThm}
would remain true if we assumed only that
$\Mon(\phi)$ has a transitive quasi-Hamiltonian subgroup, rather than a transitive Dedekind
subgroup.  Any counterexample to this generalization of Proposition~\ref{MonThm} would have $N=C=1$
in the notation of Lemma~\ref{MonLemma} (since $G:=\Mon(\phi)$ is faithful so that $\core_G(H)=1$),
but we do not know whether this can happen.
We note that the proof of Proposition~\ref{MonProp} shows that this cannot happen if every
minimal nontrivial subgroup of $A$ is normal.
\end{remark}

\begin{remark}
Proposition~\ref{MonThm} was first proved for complex polynomials
as a step in the proof of \cite[Thm.~R.2]{Muller}; in this case \cite[Thm.~2.13]{ZM} shows
that the conclusion holds if we replace the hypothesis that the decompositions are inequivalent and
complete by the hypothesis that $\gcd(\deg \phi_2,\deg\psi_2)=1=\gcd(\deg\phi_1,\deg\psi_1)$.
Theorem~\ref{MonThmfull} was first proved for complex polynomials in \cite[Thm.~1.3]{ZM}.
\end{remark}

%################################################################################
%################################################################################

\section{Automorphism groups of covers}\label{Sec:BNg}

In this section we examine the automorphism group of a cover $\phi\colon C\to D$,
and show that if the monodromy group of $\phi$ has a transitive quasi-Hamiltonian subgroup
then the collection of automorphism groups of the
indecomposable covers in a complete decomposition of $\phi$ is uniquely determined by $\phi$.

\begin{definition} If $\phi\colon C\to D$ is a cover of curves over a field $K$, then an \emph{automorphism}
of $\phi$ is an automorphism $\sigma$ of $C$ which is defined over $K$ and which satisfies
$\phi\circ\sigma=\phi$.
\end{definition}

We write $\Aut(\phi)$ to denote the set of all automorphisms of $\phi$, and note that $\Aut(\phi)$
is a group under the operation of composition.  We will prove the following result.

\begin{theorem} \label{tard}
Let $\phi\colon C\to D$ be a cover of curves over a field $K$, and suppose that the monodromy group
$\Mon(\phi)$ has a transitive quasi-Hamiltonian subgroup.
If $\phi=\phi_n\circ\phi_{n-1}\circ\dots\circ\phi_1$ and $\phi=\psi_n\circ\psi_{n-1}\circ\dots\circ\psi_1$
are complete decompositions of $\phi$,
then there is a permutation $\pi$ of
$\{1,2,\dots,n\}$ such that, for each $i$ with $1\le i\le n$, we have
$\deg\phi_i=\deg\psi_{\pi(i)}$ and $\Aut(\phi_i)\cong\Aut(\psi_{\pi(i)})$.
\end{theorem}

We begin with the following simple result.

\begin{lemma}\label{symlem}
If $\phi\colon C\to D$ is a cover of curves over a field $K$, then we can write $\phi=\phi_2\circ\phi_1$ where 
$\phi_2\colon C_1\to D$ and $\phi_1\colon C\to C_1$ are covers of curves over $K$ such that
$K(C)/K(C_1)$ is Galois with Galois group $\Aut(\phi)$.
For any expression $\phi_2=\psi_2\circ\psi_1$ as the composition of two covers, we have
$\Aut(\phi)=\Aut(\psi_1\circ\phi_1)$.  Finally,
$\Aut(\phi)\cong N_G(H)/H$, where $G$ is the monodromy group
of $\phi$ and $H$ is a one-point stabilizer of $G$.
\end{lemma}

\begin{proof}
Via the standard equivalence of categories
between curves and function fields, we see that $\Aut(\phi)$ is isomorphic to the group
of automorphisms of the function field $K(C)$ which act as the identity on $K(D)$.  In other
words, $\Aut(\phi)$ is the Galois group of the largest Galois extension $K(C)/L$ where $L$
is a field between $K(D)$ and $K(C)$.  Now let $\Omega$ be the Galois closure of $K(C)/K(D)$,
and write $G:=\Gal(\Omega/K(D))$ and $\tilde{H}:=\Gal(\Omega/K(C))$.  Then $\Aut(\phi)\cong
\Gal(K(C)/L)\cong N_G(\tilde{H})/\tilde{H}$.  Since $G$ is transitive, any one-point stabilizer $H$
of $G$ is conjugate to $\tilde{H}$ in $G$, so that $N_G(H)/H\cong N_G(\tilde{H})/\tilde{H}$.
Finally, let $\phi_2\colon C_1\to D$ and $\phi_1\colon C\to C_1$ be covers of curves over $K$
which correspond to the field extensions $L/K(D)$ and $K(C)/L$; then $\phi=\phi_2\circ\phi_1$
and $K(C)/L$ is Galois with Galois group $\Aut(\phi)=\Aut(\phi_1)=\Aut(\psi_1\circ\phi_1)$, as required.
\end{proof}

We now record an immediate geometric reformulation of the condition that $K(C)/K(C_1)$ is
Galois in the above result.

\begin{lemma}\label{fiber}
Let $\phi\colon C\to D$ be a cover of curves over a field $K$.  Then the function field extension
 $K(C)/K(D)$ is Galois if and only if the irreducible components of the fibered product $C\times_D C$
 are precisely the graphs of the functions $\nu\colon C\to C$ for $\nu\in\Aut(\phi)$.
\end{lemma}

\begin{remark} \label{faccore}
In the geomeric setting, one says that $\phi$ is Galois when $K(C)/K(D)$ is Galois.
In the algebraic setting, where $\phi$ is a polynomial $f(X)$, the above result
says (for $x$ transcendental over $K$) that $K(x)/K(f(x))$ is Galois if and only if $f(X)-f(Y)$ is a
constant times the product of all $X-\nu(Y)$ where $\nu\in K[X]$ satisfies $f\circ\nu=f$.
Polynomials with this property are called ``factorable" \cite{Cohenfac}, and if $\phi$ is a polynomial
then the polynomial
playing the role of $\phi_1$ in Lemma~\ref{symlem} is called the ``factorable core" of $\phi$.
\end{remark}

Next we show that indecomposable covers with nontrivial automorphism groups are highly restricted.

\begin{corollary}\label{gamma.Ind}
If $\phi\colon C\to D$ is an indecomposable cover of curves over a field $K$, then the following are equivalent:
\begin{enumerate}
\item $\deg\phi$ is prime and both $\Aut(\phi)$ and $\Mon(\phi)$ are cyclic of order $\deg \phi$
\item $\Mon(\phi)$ is abelian
\item $\Mon(\phi)$ is regular
\item $\lvert\Aut(\phi)\rvert>1$.
\end{enumerate}
\end{corollary}

\begin{proof}
We show that (1)$\implies$(2)$\implies$(3)$\implies$(4)$\implies$(1).  The first implication
is immediate.  Now let $\Omega$ be the Galois closure of $K(C)/K(D)$, and put
$H:=\Gal(\Omega/K(C))$.
If (2) holds then $H$ is a normal subgroup of $\Mon(\phi)$, so that
$K(C)/K(D)$ is Galois and thus $\Omega=K(C)$, whence $H=1$ so (3) holds.
If (3) holds then Lemma~\ref{symlem} implies that $\Aut(\phi)\cong\Mon(\phi)$ is nontrivial.
Finally, suppose that (4) holds.  Write $\phi=\phi_2\circ\phi_1$ as in Lemma~\ref{symlem}.
Since $\phi$ is indecomposable and $\deg\phi_1=\lvert\Aut(\phi)\rvert>1$, we must have
$\deg\phi_1=\deg\phi$.  Therefore $K(C)/K(D)$ is Galois with Galois group $\Aut(\phi)$,
so that $\Mon(\phi)\cong\Gal(K(C)/K(D))\cong\Aut(\phi)$.  Since $\phi$ is indecomposable,
Lemma~\ref{Decr.Incr} implies that $\Mon(\phi)$ has no nontrivial proper subgroups, so that
$\Mon(\phi)$ must have prime order.
\end{proof}

\begin{proof}[Proof of Theorem~\ref{tard}]
By Theorem~\ref{RittFirst} it suffices to prove the result when $n=2$.
Assuming $n=2$, let $H$ be a one-point stabilizer in $G:=\Mon(\phi)$, and let $A$ be a transitive
quasi-Hamiltonian subgroup of $G$.  Let $U$ and $V$ be groups between $H$ and $G$ which
correspond to $\phi_1$ and $\psi_1$ via Lemma~\ref{Decr.Incr}.  We assume $U\ne V$, since otherwise
the conclusion is immediate.  By symmetry, it suffices to show that $\Aut(\phi_2)\cong\Aut(\psi_1)$
and $\deg\phi_2=\deg\psi_1$.  
By Lemma~\ref{symlem}, this holds if $\Mon(\phi_2)\cong\Mon(\psi_1)$, which is the case unless
$\core_G(U)=\core_V(H)=\core_G(H)$ (by Lemmas~\ref{coremon} and \ref{MonLemma}).
Since $G$ is the Galois group of a field extension, in particular it is a faithful permutation group,
so that $\core_G(H)=1$.  Henceforth assume that $\core_G(U)=\core_V(H)=1$.
Therefore $U$ is not normal in $G$,
so Corollary~\ref{gamma.Ind} implies that $\Aut(\phi_2)=1$.  
If $\Aut(\psi_1)\ne 1$ then Corollary~\ref{gamma.Ind} implies that $H=1$, so $G=HA=A$ and thus
$U$ is a maximal proper subgroup of $A$.
This contradicts Ore's  result \cite[Thm.~5]{Oregroup} that every maximal proper subgroup of a 
quasi-Hamiltonian group is normal, so in fact $\Aut(\psi_1)=1$.  Finally,  maximality and distinctness
of the chains forces $UV=G$ and $U\cap V=H$, so that
$[G:U]=[V:H]$ and thus $\deg\phi_2=\deg\psi_1$.
\end{proof}

\begin{remark}
In the above proof we used Ore's result that every maximal proper subgroup of a quasi-Hamiltonian group
is normal.  In other words, we used a portion of the defining property of Dedekind groups which
remains true for the larger class of quasi-Hamiltonian groups.  It would be interesting to generalize
Theorem~\ref{tard} by replacing quasi-Hamiltonian groups by an even larger class of groups.
\end{remark}

\begin{remark}
Since Ore's proof is quick, we include it here for the reader's convenience.  If $U$ is a non-normal
maximal subgroup of a quasi-Hamiltonian group $G$ then $U$ has a conjugate $V\ne U$, and plainly
$V$ cannot be a subgroup of $U$.
Therefore $UV$ is a group which strictly contains $U$, so maximality implies
that $UV=G$, whence $V=v^{-1}Uv$ for some $v\in V$, yielding the contradiction $V\ne U=vVv^{-1}=V$.
\end{remark}

In the remainder of this section we describe some ways in which the concept of the automorphism
group of a cover has arisen (in the special case of polynomials or rational functions)
in the literature on complex dynamics and value sets of polynomials over finite fields.
We note that previous authors have had to work much harder in order to prove some of the above results
in the case of polynomials, and their proofs do not extend to treat more general covers.
  For $f(X)\in K(X)$, the group $\Aut(f)$ consists of the degree-one rational
functions $\mu(X)\in K(X)$ for which $f\circ\mu=f$; if $f(X)\in K[X]\setminus K$ then every element of
 $\Aut(f)$ must permute the set $f^{-1}(\infty)=\{\infty\}$, and hence must be a degree-one polynomial.

In case $f(X)\in\C[X]$ has degree $n\ge 2$, we write $\Gamma(f)$ for the group of degree-one
$\mu\in\C[X]$ for which there exists a degree-one $\nu\in\C[X]$ satisfying $\nu\circ f\circ\mu = f$.
One can easily check that $\Gamma(f)$ is infinite if and only if $f(X)$ is conjugate to $X^n$, and
that if $\Gamma(f)$ is finite then it is cyclic of order $m$; here, for any degree-one $\theta\in\C[X]$
such that $\hat f:=\theta^{-1}\circ f\circ\theta$ has no term of degree $n-1$, the number $m$ is the greatest
common divisor of the collection of all differences between degrees of pairs of terms of $\hat f$.
In case $\Gamma(f)$ is finite, it coincides with the group of symmetries of the Julia set of $f(X)$
\cite{Beardonsymm}.   The paper \cite{BNg} studies the subgroup $\Aut(f)$ of $\Gamma(f)$, and
gives a lengthy proof of Theorem~\ref{tard} in the special case of complex polynomials by making use
of the much more difficult ``Ritt's Second Theorem" (which should not be confused with Ritt's First Theorem
as discussed in Remark~\ref{Rittrem}).  In a subsequent paper we will discuss the analogue
of $\Gamma(f)$ for arbitrary covers of curves (over an arbitrary field), and in the case of rational
functions we will discuss the union of the groups $\Gamma(f^{\circ n})$ where $f^{\circ n}$ denotes
the $n$-th iterate of $f(X)$.  Some results in this direction appear in \cite{Levin}.

Now consider $f(X)\in K[X]\setminus K[X^p]$, where $K$ is a field of characteristic $p\ge 0$.
As noted in Remark~\ref{faccore}, we can write $f=g\circ h$ where $g,h\in K[X]$ and
$K(x)/K(h(x))$ is Galois with Galois group $\Aut(f)$.
  Now $\Aut(f)$ consists of all degree-one $\mu\in K[X]$ for which $X-\mu(Y)$ divides 
$f(X)-f(Y)$, so in particular $h(X)-h(Y)$ is the product of degree-one polynomials in $K[X,Y]$.
Together with the known list of
 possibilities for $h(X)$  \cite[Thm.~1]{TS68}, this implies the main results of \cite{Cohenfac}.
The elements of $\Aut(f)$ play a prominent role in the study of the image set $f(K)$ where $K$ is
a finite field.  This classical topic has a rich tradition, with important contributions by
Betti, Mathieu, Hermite, Dickson, Schur, Davenport, Carlitz, Birch, Swinnerton-Dyer, Mordell, Bombieri,
and many others.  The group $\Aut(f)$ plays an especially fundamental role in case the ratio
$\lvert f(K)\rvert/\lvert K\rvert$ is either
large \cite{Hayes, Lenstra} or small \cite{GCM,Mitkin}, and also arises in other well-behaved
cases \cite{Cohenvalset,Williamsextreme}.  See \cite{MullenZ} for a recent survey on this topic.

%################################################################################
%################################################################################

\section{A divisibility property of automorphism groups of subcovers} \label{secdiv}

In this section we prove the following divisibility result.

\begin{theorem}\label{ThmGS}
Let $\phi$ be a cover of curves over a field $K$, and assume that the monodromy group
$\Mon(\phi)$ has a transitive Dedekind subgroup.
Let $\phi=\phi_n\circ\phi_{n-1}\circ\dots\circ\phi_1$ be a decomposition of $\phi$,
and for each $i=1,2,\dots,n$ write $\psi_i:=\phi_i\circ\phi_{i-1}\circ\dots\circ\phi_1$.
Then for $1\le i<n$ the group $\Aut(\psi_i)$ is a normal subgroup of $\Aut(\psi_{i+1})$,
and the quotient group is isomorphic to a subgroup of $\Aut(\phi_{i+1})$.
In particular, 
\[
\lvert\Aut(\phi)\rvert \text{ divides }\prod_{i=1}^n\lvert\Aut(\phi_i)\rvert.
\]
\end{theorem}

The crucial case in the proof of Theorem~\ref{ThmGS} is contained in the following result
which is of independent interest.

\begin{prop}\label{GS2}
Let $\theta\colon C_1\to D$ and $\rho\colon C\to C_1$  be
covers of curves over $K$ for which the monodromy group of $\psi:=\theta\circ\rho$ has a
transitive Dedekind subgroup $A$.
Then, for each $\mu\in\Aut(\psi)$, there is a unique $\nu\in\Aut(\theta)$ for which
$\rho\circ\mu=\nu\circ\rho$.  Moreover, the map $\mu\mapsto\nu$ defines a homomorphism
$\Aut(\psi)\to\Aut(\theta)$ with kernel $\Aut(\rho)$.
\end{prop}

\begin{proof}
Let $G>U>H$ be the chain of groups corresponding to the decomposition $\psi=\theta\circ\rho$,
where $H$ is a one-point stabilizer of $G:=\Mon(\psi)$.  Then $J:=N_G(H)\cap A$ and $J_1:=U\cap A$
satisfy $N_G(H)=HJ$ and $U=HJ_1$.
Since $J\le N_G(H)$ and $J$ normalizes $J_1$ (because $A$ is Dedekind), it follows that
$J\le N_G(U)$, so also $N_G(H)=HJ\le N_G(U)$ (since $H\le U$).
Thus $N_G(H)U$ is a group which normalizes $U$.  Let $\theta=\theta_2\circ\theta_1$ be the chain of
covers which corresponds to the chain of groups $G\ge N_G(H)U\ge U$.
Writing $\theta_1$ as $\theta_1\colon C_1\to E$, we see that $K(C_1)/K(E)$ is Galois, so that
$\Aut(\theta_1)\cong\Mon(\theta_1)\cong N_G(H)U/U$ has order equal to $\deg\theta_1$.
By Lemma~\ref{fiber}, the irreducible components of the fibered product $Z:=C_1\times_E C_1$
are the graphs of the functions $\nu\colon C_1\to C_1$ for $\nu\in\Aut(\theta_1)$.
Lemma~\ref{symlem} implies that $\Aut(\psi)=\Aut(\theta_1\circ\rho)$, so for $\mu\in\Aut(\psi)$
we have $\theta_1\circ\rho\circ\mu=\theta_1\circ\rho$.  Therefore the map 
$P\mapsto (\rho\circ\mu(P), \rho(P))$ is a morphism $C_0\to Z$, so its image is a component
of $Z$, whence there is some $\nu\in\Aut(\theta_1)$ for which $\rho\circ\mu=\nu\circ\rho$.
Now it is clear that the map $\mu\mapsto\nu$ is a homomorphism $\Aut(\psi)\to\Aut(\theta_1)$
with kernel $\Aut(\rho)$, which proves the result since $\Aut(\theta_1)\le\Aut(\theta)$.
\end{proof}

\begin{proof}[Proof of Theorem~\ref{ThmGS}]
Since $\Mon(\phi)$ has a transitive Dedekind subgroup, so too does $\Mon(\psi_{i+1})$
for any $i<n$.  Apply Proposition~\ref{GS2} with $(\psi,\theta,\rho)=(\psi_{i+1},\phi_{i+1},\psi_i)$
to conclude that $\Aut(\psi_i)\trianglelefteq\Aut(\psi_{i+1})$ and
$\Aut(\psi_{i+1})/\Aut(\psi_i)$ is isomorphic to a subgroup of $\Aut(\phi_{i+1})$.
It follows that
\[
\lvert\Aut(\phi)\rvert = \lvert\Aut(\psi_1)\rvert\cdot\prod_{i=1}^{n-1} \frac{
\lvert\Aut(\psi_{i+1})\rvert}{\lvert\Aut(\psi_i)\rvert}
\]
divides $\lvert\Aut(\phi_1)\rvert\cdot\prod_{i=1}^{n-1}\lvert\Aut(\phi_{i+1})\rvert$,
which concludes the proof.
\end{proof}

\begin{remark}
Theorem~\ref{ThmGS} would not remain true if we weakened the hypothesis to require
only that $\Mon(\phi)$ contains a transitive quasi-Hamiltonian subgroup.
For, let $G$ be any finite quasi-Hamiltonian group which is not a Dedekind group, and let
$U$ be a non-normal subgroup of $G$.  Let $L/\C(x)$ be a Galois extension with group $G$
(this exists for any choice of $G$, essentially by Riemann's existence theorem).  Then the
chain of fields $\C(x)\subseteq L^U\subseteq L$ corresponds to a chain of covers
$\phi=\phi_2\circ\phi_1$ where $\Aut(\phi_1)\cong U$ and $\Aut(\phi)\cong G$ but
$\Aut(\phi_2)\cong N_G(U)/U$ has order strictly less than $[G:U]$.
\end{remark}

\begin{remark} \label{finalsay}
The final assertion in Theorem~\ref{ThmGS} has been considered previously when $\phi$ is
given by a polynomial $f(X)$.  This assertion was proved in \cite[Thm.~1.2]{BNg} for $f(X)\in\C[X]$,
and in \cite[Thm.~8] {GS06} for $f(X)\in K[X]$ where $K$ is a field and $\charp(K)\nmid\deg f$.
The proofs in those papers apply only to polynomials, and make no mention of monodromy groups.
This difference in perspective perhaps accounts for the conjecture by Gutierrez and 
Sevilla \cite[Conj.~1]{GS06} that the final assertion in Theorem~\ref{ThmGS} is true whenever
$\phi$ is a rational function $f(X)\in K(X)$ over a field $K$ for which $\charp(K)\nmid\deg f$.  One counterexample
is $f(X):=X^3+X^{-3}$ over the field $K=\C$: for $f_2(X):=X^3-3X$ and $f_1(X):=X+X^{-1}$
we have $f=f_2\circ f_1$ but $\lvert\Aut(f)\rvert=6$ does not divide
$\lvert\Aut(f_2)\rvert\cdot\lvert\Aut(f_1)\rvert=1\cdot 2$.  There are many further
counterexamples, and indeed the perspective of the present paper explains why one should not expect
the conjecture to be true.
\end{remark}

%################################################################################
%################################################################################
%################################################################################
%################################################################################
%################################################################################

\end{document}